\documentclass[12pt,a4paper, longbibliography]{article}\UseRawInputEncoding
\usepackage[utf8]{inputenc}
\usepackage[T1]{fontenc}
\usepackage{authblk}

\usepackage{amsmath, amscd, amsthm, amscd, amsfonts, amssymb, graphicx, color, soul, enumerate, xcolor,  mathrsfs, latexsym, bigstrut, framed, caption}
\usepackage[usenames,dvipsnames]{pstricks}
\usepackage{pstricks-add}
\usepackage{epsfig}
\usepackage{pst-grad} 
\usepackage{pst-plot} 
\usepackage[space]{grffile} 
\usepackage{etoolbox} 
\makeatletter 
\patchcmd\Gread@eps{\@inputcheck#1 }{\@inputcheck"#1"\relax}{}{}
\makeatother
\usepackage{pstricks,pst-node,pst-tree}
\usepackage[bookmarksnumbered, colorlinks, plainpages,backref]{hyperref}
\hypersetup{colorlinks=true,linkcolor=blue, anchorcolor=green, citecolor=midblue, urlcolor=darkblue, filecolor=magenta, pdftoolbar=true}
\usepackage[all]{xy}

\usepackage{tikz}
\usetikzlibrary{calc,decorations.pathreplacing,decorations.markings,positioning,shapes}

\newcommand{\cp}{\ensuremath{\mathbin{\Box}}}
\newcommand{\Aut}{\ensuremath{\operatorname{Aut}}}
\newcommand{\sym}{\mathrm{Sym}}
\newcommand{\End}{\ensuremath{\operatorname{End}}}
\newcommand{\Prob}{\ensuremath{\operatorname{Pr}}}
\newcommand{\id}{\ensuremath{\text{\rm id}}}
\newcommand{\dist}{\ensuremath{\operatorname{dist}}}
\newcommand{\cg}{\overline{G}}
\newcommand{\D}{\mathrm{D}}
\newcommand{\strong}{\boxtimes}
\newcommand{\supp}{\mathrm{Supp}}
\newcommand\stackplus[1]{\makebox[0ex][l]{$+$} \raisebox{-.75ex}{\makebox[2ex]{$_{#1}$}}}
\newcommand{\rdi}{\mu}
\newcommand {\md}{{\rm mod}\,}
\newcommand{\gst}{\gamma_{\rm st}}
\newcommand\F{F}
\newcommand\Pf{Pf}
\DeclareMathOperator{\Frob}{Frob}
\def\ld{\mbox{ld}\,}
\definecolor{cupgreen}{rgb}{0,0.498,0.208}
\definecolor{cupblue}{rgb}{0,0,.5}
\definecolor{cupred}{rgb}{1,0.04,0}
\definecolor{cuppink}{rgb}{0.925,0,0.545}
\definecolor{cupmagenta}{rgb}{0.624,0.161,0.424}
\definecolor{cupbrown}{rgb}{0.71,0.212,0.133}

\definecolor{cupgreen}{rgb}{0,0,0}
\definecolor{cupblue}{rgb}{0,0,0}
\definecolor{cupred}{rgb}{0,0,0}
\definecolor{cuppink}{rgb}{0,0,0}
\definecolor{cupmagenta}{rgb}{0,0,0}
\definecolor{cupbrown}{rgb}{0,0,0}

\definecolor{TITLE}{rgb}{0,0,0}
\definecolor{midblue}{rgb}{0.00,0.0,0.80}
\definecolor{darkblue}{rgb}{0.00,0.00,0.45}
\definecolor{SECTION}{rgb}{0.50,0.00,1.00}
\definecolor{THM}{rgb}{0.8,0,0.1}
\definecolor{SEC}{rgb}{0,0,1}
\newcommand{\A}{\mathcal A}
\newcommand{\B}{\mathcal B}
\newcommand{\M}{\mathcal M}
\newcommand{\N}{\mathcal N}
\newcommand{\p}{\mathfrak p}
\newcommand{\q}{\mathfrak q}

\newcommand{\aut}{\mathrm{Aut}}
\newcommand{\m}{\mathrm{M}}
\newcommand{\stab}{\mathrm{Stab}}

\makeatother
\renewcommand{\baselinestretch}{1.05}\normalsize
\newtheorem{theorem}{{\color{THM} Theorem}}[section]

\DeclareRobustCommand{\stirling}{\genfrac\{\}{0pt}{}}

\newtheorem{corollary}[theorem]{{\color{THM}Corollary}}

\theoremstyle{definition}

\newtheorem{remark}[theorem]{{\color{THM}Remark}}
\numberwithin{equation}{section}

\renewcommand{\refname}{{\color{SEC} References}}
\newcommand{\f}{\rightarrow}
\newcommand{\la}{\langle}
\newcommand{\ra}{\rangle}
\date{}
\title{Strong domination number of Haj\'{o}s sum and vertex-sum of two graphs}

\author[1]{{ Nima Ghanbari\thanks{ E-mail: nima.ghanbari@uib.no}}}
\affil[1]{{\small Department of Informatics, University of Bergen, P.O. Box 7803, 5020 Bergen, Norway}}

\author[2]{{ Saeid Alikhani \thanks{Corresponding author.~Email: alikhani@yazd.ac.ir}}}
\affil[2]{{\small Department of Mathematical Sciences, Yazd University, Yazd, Iran}}

\begin{document}

	\maketitle
\begin{abstract}
  Let $G=(V,E)$ be a simple graph. A set $D\subseteq V$ is a strong dominating set of $G$, if for every vertex $x\in V\setminus D$ there is a vertex $y\in D$ with $xy\in E(G)$ and $\deg(x)\leq \deg(y)$. The strong domination number $\gst(G)$ is defined as the minimum cardinality of a strong dominating set.  In this paper, we study the strong domination number of Haj\'{o}s sum and vertex-sum of two graphs. 
\end{abstract}


\medskip
\noindent{\bf AMS Subj.\ Class.: } 05C15, 05C25.

\section{Introduction}

A dominating set of a graph $G=(V,E)$ is a subset $D$ of $V$ such that every vertex in $V\setminus D$ is adjacent to at least one member of $D$.  
The minimum cardinality of all dominating sets of $G$ is called  the  domination number of $G$ and is denoted by $\gamma(G)$. This parameter has  been extensively studied in the literature and there are  hundreds of papers concerned with domination.  
We recommend a fundamental book \cite{domination} about domination in general. 
The various different domination concepts are
well-studied now, however new concepts are introduced frequently and the interest is growing
rapidly.

 A set $D\subseteq V$ is a \emph{strong dominating set} of $G$, if for every vertex $x\in  \overline{D}=V\setminus D$ there is a vertex $y\in D$ with $xy\in E(G)$ and $\deg(x)\leq \deg(y)$. The \emph{strong domination number} $\gst(G)$ is defined as the minimum cardinality of a strong dominating set. A $\gst$-\emph{set} of $G$ is a strong dominating set of $G$ of minimum cardinality $\gst(G)$. If $D$ is a strong dominating set in a graph $G$, then we say that a vertex $u \in \overline{D}$ is \emph{strong dominated} by a vertex $v \in D$ if $uv\in E(G)$, and $\deg(u)\leq \deg(v)$.

The strong domination number was introduced in \cite{DM} and some upper bounds on this parameter presented in \cite{DM2,DM}. Similar to strong domination number, a set $D\subset V$  is a weak  dominating set of $G$, if every vertex $v\in V\setminus S$  is
adjacent to a vertex $u\in D$ such that $deg(v)\geq deg(u)$ (see \cite{Boutrig}). The minimum cardinality of a weak dominating set of $G$ is denoted by $\gamma_w(G)$. Boutrig and  Chellali proved that the relation $\gamma_w(G)+\frac{3}{\Delta+1}\gamma_{st}(G)\leq n$ holds for any connected graph of order $n\geq 3.$ Alikhani, Ghanbari and Zaherifard \cite{sub} examined the effects on $\gamma_{st}(G)$ when $G$ is modified by the edge deletion, the edge subdivision and the edge contraction. Also they studied the strong domination number of $k$-subdivision of $G$.   

Motivated by enumerating of  the number of dominating sets of a graph and domination polynomial (see e.g. \cite{euro,saeid1}), the enumeration of  the strong dominating sets for certain
graphs has studied in \cite{JAS}. 
Study of the strong domination number of graph operation is a natural and interesting subject and for join and corona product has studied (\cite{JAS}).  
In this paper, we consider other kinds of graph operations which are  called Haj\'{o}s sum and vertex sum of two graphs. The Haj\'{o}s sum  is useful
when either of the network is disrupted and certain node(s)
is(are) not functioning. Then that node(s) is(are) to be identified(fused) with the node of a network which is functioning
properly and thus new network is constructed.

\section{ Haj\'{o}s sum}

In this section, we study the strong domination number of  Haj\'{o}s sum of two graphs. First we recall its definition. 
Given graphs $G_1 = (V_1,E_1)$ and $G_2 = (V_2, E_2)$
with disjoint vertex sets, an edge $x_1y_1\in E_1$, and an edge $x_2y_2\in E_2$, the
\emph{Haj\'{o}s sum} $G_3 = G_1(x_1y_1) +_H G_2(x_2y_2)$ is the graph obtained as follows:
begin with $G_3 = (V_1 \cup V_2, E_1 \cup E_2)$; then in $G_3$ delete the edges $x_1y_1$ and
$x_2y_2$, identify the vertices $x_1$ and $x_2$ as $v_H(x_1x_2)$, and add the edge $y_1y_2$~\cite{HAJOSSUM}.
Figure~\ref{HaJ-K6C6} shows the Haj\'{o}s sum  of $K_6$ and $C_6$ 
with respect to $x_1y_1$ and $x_2y_2$.

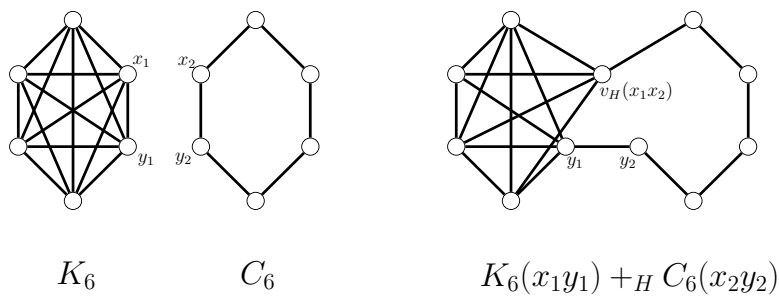
\begin{figure}
\begin{center}
\psscalebox{0.6 0.6}
{
\begin{pspicture}(0,-6.9293056)(16.402779,-0.66791654)
\psline[linecolor=black, linewidth=0.06](2.601389,-2.0693054)(2.601389,-3.6693053)(1.4013889,-4.8693056)(0.20138885,-3.6693053)(0.20138885,-2.0693054)(1.4013889,-0.86930543)(2.601389,-2.0693054)(2.601389,-2.0693054)
\psline[linecolor=black, linewidth=0.06](1.4013889,-0.86930543)(0.20138885,-3.6693053)(2.601389,-3.6693053)(1.4013889,-0.86930543)(1.4013889,-0.86930543)
\psline[linecolor=black, linewidth=0.06](0.20138885,-2.0693054)(2.601389,-2.0693054)(1.4013889,-4.8693056)(0.20138885,-2.0693054)(0.20138885,-2.0693054)
\psline[linecolor=black, linewidth=0.06](0.20138885,-3.6693053)(2.601389,-2.0693054)(2.601389,-2.0693054)
\psline[linecolor=black, linewidth=0.06](0.20138885,-2.0693054)(2.601389,-3.6693053)(2.601389,-3.6693053)
\psline[linecolor=black, linewidth=0.06](4.201389,-2.0693054)(5.4013886,-0.86930543)(6.601389,-2.0693054)(6.601389,-3.6693053)(5.4013886,-4.8693056)(4.201389,-3.6693053)(4.201389,-2.0693054)(4.201389,-2.0693054)
\psdots[linecolor=black, dotstyle=o, dotsize=0.4, fillcolor=white](0.20138885,-2.0693054)
\psdots[linecolor=black, dotstyle=o, dotsize=0.4, fillcolor=white](2.601389,-2.0693054)
\psdots[linecolor=black, dotstyle=o, dotsize=0.4, fillcolor=white](0.20138885,-3.6693053)
\psdots[linecolor=black, dotstyle=o, dotsize=0.4, fillcolor=white](2.601389,-3.6693053)
\psdots[linecolor=black, dotstyle=o, dotsize=0.4, fillcolor=white](5.4013886,-0.86930543)
\psdots[linecolor=black, dotstyle=o, dotsize=0.4, fillcolor=white](4.201389,-2.0693054)
\psdots[linecolor=black, dotstyle=o, dotsize=0.4, fillcolor=white](6.601389,-2.0693054)
\psdots[linecolor=black, dotstyle=o, dotsize=0.4, fillcolor=white](6.601389,-3.6693053)
\psdots[linecolor=black, dotstyle=o, dotsize=0.4, fillcolor=white](5.4013886,-4.8693056)
\psdots[linecolor=black, dotstyle=o, dotsize=0.4, fillcolor=white](4.201389,-3.6693053)
\rput[bl](1.0413889,-6.809305){\LARGE{$K_6$}}
\rput[bl](5.061389,-6.809305){\LARGE{$C_6$}}
\rput[bl](10.321389,-6.9293056){\LARGE{$K_6(x_1y_1)+_H C_6(x_2y_2)$}}
\rput[bl](12.941389,-2.6493053){$v_H(x_1x_2)$}
\rput[bl](2.7013888,-1.9093055){$x_1$}
\rput[bl](3.7013888,-1.9293054){$x_2$}
\rput[bl](3.641389,-4.0693054){$y_2$}
\rput[bl](2.821389,-4.0893054){$y_1$}
\psline[linecolor=black, linewidth=0.06](11.001389,-0.86930543)(9.801389,-3.6693053)(12.201389,-3.6693053)(11.001389,-0.86930543)(11.001389,-0.86930543)
\psline[linecolor=black, linewidth=0.06](9.801389,-2.0693054)(12.201389,-3.6693053)(12.201389,-3.6693053)
\psline[linecolor=black, linewidth=0.06](1.4013889,-0.86930543)(1.4013889,-4.8693056)(1.4013889,-4.8693056)
\psline[linecolor=black, linewidth=0.06](11.001389,-0.86930543)(11.001389,-4.8693056)(11.001389,-4.8693056)
\psline[linecolor=black, linewidth=0.06](11.001389,-0.86930543)(9.801389,-2.0693054)(9.801389,-3.6693053)(11.001389,-4.8693056)(11.001389,-4.8693056)
\psline[linecolor=black, linewidth=0.06](11.001389,-4.8693056)(12.201389,-3.6693053)(12.201389,-3.6693053)
\psline[linecolor=black, linewidth=0.06](13.801389,-3.6693053)(15.001389,-4.8693056)(16.20139,-3.6693053)(16.20139,-2.0693054)(15.001389,-0.86930543)(15.001389,-0.86930543)
\psline[linecolor=black, linewidth=0.06](12.201389,-3.6693053)(13.801389,-3.6693053)(13.801389,-3.6693053)
\psline[linecolor=black, linewidth=0.06](15.001389,-0.86930543)(13.001389,-2.0693054)(13.001389,-2.0693054)
\psline[linecolor=black, linewidth=0.06](11.001389,-0.86930543)(13.001389,-2.0693054)(13.001389,-2.0693054)
\psline[linecolor=black, linewidth=0.06](9.801389,-2.0693054)(13.001389,-2.0693054)(9.801389,-3.6693053)(9.801389,-3.6693053)
\psline[linecolor=black, linewidth=0.06](13.001389,-2.0693054)(11.001389,-4.8693056)(11.001389,-4.8693056)
\psdots[linecolor=black, dotstyle=o, dotsize=0.4, fillcolor=white](13.001389,-2.0693054)
\psdots[linecolor=black, dotstyle=o, dotsize=0.4, fillcolor=white](12.201389,-3.6693053)
\psdots[linecolor=black, dotstyle=o, dotsize=0.4, fillcolor=white](13.801389,-3.6693053)
\psdots[linecolor=black, dotstyle=o, dotsize=0.4, fillcolor=white](11.001389,-4.8693056)
\psdots[linecolor=black, dotstyle=o, dotsize=0.4, fillcolor=white](9.801389,-3.6693053)
\psdots[linecolor=black, dotstyle=o, dotsize=0.4, fillcolor=white](9.801389,-2.0693054)
\psdots[linecolor=black, dotstyle=o, dotsize=0.4, fillcolor=white](11.001389,-0.86930543)
\psdots[linecolor=black, dotstyle=o, dotsize=0.4, fillcolor=white](15.001389,-0.86930543)
\psdots[linecolor=black, dotstyle=o, dotsize=0.4, fillcolor=white](16.20139,-2.0693054)
\psdots[linecolor=black, dotstyle=o, dotsize=0.4, fillcolor=white](16.20139,-3.6693053)
\psdots[linecolor=black, dotstyle=o, dotsize=0.4, fillcolor=white](15.001389,-4.8693056)
\psdots[linecolor=black, dotstyle=o, dotsize=0.4, fillcolor=white](1.4013889,-0.86930543)
\psdots[linecolor=black, dotstyle=o, dotsize=0.4, fillcolor=white](1.4013889,-4.8693056)
\rput[bl](12.221389,-4.1693053){$y_1$}
\rput[bl](13.381389,-4.1693053){$y_2$}
\end{pspicture}
}
\end{center}
\caption{Haj\'{o}s construction of $K_6$ and $C_6$.} \label{HaJ-K6C6}
\end{figure}

The following theorem gives the lower bound and the upper bound for the strong domination number of  Haj\'{o}s sum of two graphs.

	\begin{theorem}\label{thm:Hajos}
Let $G_1=(V_1,E_1)$ and $G_2=(V_2,E_2)$ be two graphs with disjoint 
vertex sets, $x_1y_1\in E_1$ and $x_2y_2\in E_2$. Also, suppose that $x_1$ and $x_2$ are not pendant vertices. Then for the Haj\'{o}s sum 
$$G_3=G_1(x_1y_1)+_H G_2(x_2y_2),$$
we have:
$$ \gst (G_1) + \gst (G_2) -\deg(x_1)-\deg(x_2)+2 \leq  \gst(G_3) 
\leq  \gst (G_1) + \gst (G_2)+1.  $$	
	\end{theorem}

	\begin{proof}
First we find the upper bound. Since $x_1$ and $x_2$ are not pendant vertices, then by the definition of the Haj\'{o}s sum we know that $\deg (v_H(x_1x_2))=\deg (x_1)+\deg (x_2)-2$. Also, $\deg _{G_3}(y_1)=\deg _{G_1}(y_1)$, and $\deg _{G_3}(y_2)=\deg _{G_2}(y_2)$. Suppose that $D_i$ is a $\gst$-set of $G_i$, for $i=1,2$.  We have the following cases:
\begin{itemize}
\item[(i)]
$y_1$ is strong dominated by $x_1$, and $y_2$ is strong dominated by $x_2$. Without loss of generality, suppose that $\deg(y_1)\geq \deg(y_2)$. Let 
$$D_3=\left(D_1\setminus \{x_1\}\right)\cup \left(D_2\setminus \{x_2\}\right)\cup \{v_H(x_1x_2),y_1\}.$$
$D_3$ is a strong dominating set of $G_3$, because $y_2$ is strong dominated by $y_1$, and every other vertices in $\overline{D_3}$ is strong dominated by the same vertices as before or $v_H(x_1x_2)$. So we have 
$$\gst(G_3) \leq  \gst (G_1) + \gst (G_2).$$

\item[(ii)]
$y_1$ is strong dominated by $x_1$, and $y_2$ is not strong dominated by $x_2$. In this case, we may have $y_2\in D_2$ or $y_2\in \overline{D_2}$, and we may have $x_2\in D_2$ or $x_2\in \overline{D_2}$. Let 
$$D_3=\left(D_1\setminus \{x_1\}\right)\cup \left(D_2\setminus \{x_2\}\right)\cup \{v_H(x_1x_2),y_1\}.$$
$D_3$ is a strong dominating set of $G_3$, because if $y_2\in \overline{D_2}$, then it is strong dominated by the same vertex as before, and every other vertices in $\overline{D_3}$ is strong dominated by the same vertices as before or $v_H(x_1x_2)$. So, in the worst case, which is $x_2\in \overline{D_2}$, we have 
$$\gst(G_3) \leq  \gst (G_1) + \gst (G_2)+1.$$

\item[(iii)]
$y_1$ is not strong dominated by $x_1$, and $y_2$ is not strong dominated by $x_2$.
By a similar discussion as part (ii), 
$$D_3=\left(D_1\setminus \{x_1\}\right)\cup \left(D_2\setminus \{x_2\}\right)\cup \{v_H(x_1x_2)\},$$
is a strong dominating set of $G_3$, and in the worst case, we have 
$$\gst(G_3) \leq  \gst (G_1) + \gst (G_2)+1.$$

\item[(iv)]
$x_1$ is strong dominated by $y_1$, and $x_2$ is strong dominated by $y_2$. Then clearly 
$$D_3=D_1\cup D_2\cup \{v_H(x_1x_2)\},$$
is a strong dominating set of $G_3$, and we have 
$$\gst(G_3) \leq  \gst (G_1) + \gst (G_2)+1.$$

\item[(v)]
$x_1$ is strong dominated by $y_1$, and $x_2$ is not strong dominated by $y_2$. Then we may have $y_2$ is strong dominated by $x_2$, which we have the result by similar argument as case (ii). Otherwise, by a similar argument as part (ii), 
$$D_3=\left(D_1\cup D_2\setminus \{x_2\}\right)\cup \{v_H(x_1x_2)\}.$$
is a strong dominating set of $G_3$, and in the worst case, we have 
$$\gst(G_3) \leq  \gst (G_1) + \gst (G_2)+1.$$

\item[(vi)]
$x_1$ is not strong dominated by $y_1$, and $x_2$ is not strong dominated by $y_2$.  Then we may have
$y_1$ is strong dominated by $x_1$, and $y_2$ is strong dominated by $x_2$, which gives us the result by case (i), or we may have $y_1$ is strong dominated by $x_1$, and $y_2$ is not strong dominated by $x_2$, which gives us the result by case (ii). Otherwise,
by similar argument as before, 
$$D_3=\left(D_1\setminus \{x_1\}\right)\cup \left(D_2\setminus \{x_2\}\right)\cup \{v_H(x_1x_2)\},$$
is a strong dominating set of $G_3$, and in the worst case, we have 
$$\gst(G_3) \leq  \gst (G_1) + \gst (G_2)+1.$$
\end{itemize}
So, in general, we have $\gst(G_3) \leq  \gst (G_1) + \gst (G_2)+1$. Now, we find the lower bound. Suppose that $S_3$ is a $\gst$-set of $G_3$. We find strong dominating sets of $G_1$ and $G_2$ based on $S_3$.  We consider the following cases:
\begin{itemize}
\item[(i)]
$v_H(x_1x_2)\in S_3$. Here we consider the following sub-cases:
\begin{itemize}
\item[(a)]
$y_1\in S_3$ and $y_2\in S_3$. If $v_H(x_1x_2)$ is not strong dominating any vertices in $\overline{S_3}$, then
$$S_1=\left(S_3\setminus \Big( V(G_2)\cup \{v_H(x_1x_2)\} \Big)\right)\cup \{x_1\}$$ 
is a strong dominating set of $G_1$, and  
$$S_2=\left(S_3\setminus \Big( V(G_1)\cup \{v_H(x_1x_2)\} \Big)\right)\cup \{x_2\}$$
is a strong dominating set of $G_2$. But, if $v_H(x_1x_2)$ is  strong dominating some vertices in $\overline{S_3}$, then after forming $G_1$ and $G_2$ from $G_3$, then if 
$\deg(x_1)\geq \max \{\deg(u)~|~ u\in N(x_1)\}$,
and 
$\deg(x_2)\geq \max \{\deg(v)~|~ v\in N(x_2)\}$,
we consider $S_1$ and $S_2$ as mentioned. If 
$\deg(x_1)\geq \max \{\deg(u)~|~ u\in N(x_1)\}$,
But
$\deg(x_2)\ngeqslant \max \{\deg(v)~|~ v\in N(x_2)\}$,
we consider $S_1$  as mentioned, and let 
$$S_2=\left(S_3\setminus \Big( V(G_1)\cup \{v_H(x_1x_2)\} \Big)\right)\cup N(x_2),$$
then one can easily check that $S_2$ is a strong dominating set of $G_2$. If 
$\deg(x_1)\ngeqslant \max \{\deg(u)~|~ u\in N(x_1)\}$,
and
$\deg(x_2)\ngeqslant \max \{\deg(v)~|~ v\in N(x_2)\}$,
we consider 
$$S_1=\left(S_3\setminus \Big( V(G_2)\cup \{v_H(x_1x_2)\} \Big)\right)\cup N(x_1),$$
and
$$S_2=\left(S_3\setminus \Big( V(G_1)\cup \{v_H(x_1x_2)\} \Big)\right)\cup N(x_2).$$
Here $S_1$ and $S_2$ are strong dominating sets of $G_1$ and $G_2$, respectively. So in the worst case we have 
$$\gst(G_1) + \gst (G_2) \leq \gst (G_3)-1 +\deg(x_1)-1+\deg(x_2)-1.$$
\item[(b)]
$y_1\in S_3$ and $y_2\notin S_3$.  If $v_H(x_1x_2)$ is not strong dominating any vertices in $\overline{S_3}$, then one can easily check that
$$S_1=\left(S_3\setminus \Big( V(G_2)\cup \{v_H(x_1x_2)\} \Big)\right)\cup \{x_1\}$$ 
is a strong dominating set of $G_1$, and  one of the 
$$S_2=\left(S_3\setminus \Big( V(G_1)\cup \{v_H(x_1x_2)\} \Big)\right)\cup \{x_2\},$$
or
$$S_2'=\left(S_3\setminus \Big( V(G_1)\cup \{v_H(x_1x_2)\} \Big)\right)\cup \{y_2\}$$
 is a strong dominating set of $G_2$ (or possibly both are strong dominating sets of $G_2$). Otherwise, by similar argument as part (a), we conclude that 
$$\gst(G_1) + \gst (G_2) \leq \gst (G_3)-1 +\deg(x_1)-1+\deg(x_2).$$
\item[(c)]
$y_1\notin S_3$ and $y_2\notin S_3$. Then there exists $y_1'\in V(G_1)$ such that $y_1$ is strong dominated by that, and there exists $y_2'\in V(G_2)$ and is strong dominating $y_2$. Then one can easily check that 
$$S_1=\left(S_3\setminus \Big( V(G_2)\cup \{v_H(x_1x_2)\} \Big)\right)\cup \{x_1\}$$ 
is a strong dominating set of $G_1$, and  
$$S_2=\left(S_3\setminus \Big( V(G_1)\cup \{v_H(x_1x_2)\} \Big)\right)\cup \{x_2\}$$
is a strong dominating set of $G_2$, and we have 
$$\gst(G_1) + \gst (G_2) \leq \gst (G_3)+1.$$
\end{itemize}

\item[(ii)]
$v_H(x_1x_2)\notin S_3$. Without loss of generality, suppose that there exists $x_1'\in V(G_1)$ such that $v_H(x_1x_2)$ is strong dominated by $x_1'$. We consider the following cases:
\begin{itemize}
\item[(a)]
$y_1\in S_3$ and $y_2\in S_3$. Then one can easily check that 
$$S_1=S_3\setminus  V(G_2) $$ 
is a strong dominating set of $G_1$, and  
$$S_2=\left(S_3\setminus  V(G_1)\right)\cup \{x_2\}$$
is a strong dominating set of $G_2$. So
$$\gst(G_1) + \gst (G_2) \leq \gst (G_3)+1.$$
\item[(b)]
$y_1\in S_3$ and $y_2\notin S_3$. Then 
$$S_1=S_3\setminus  V(G_2) $$ 
is a strong dominating set of $G_1$, and  
$$S_2=\left(S_3\setminus  V(G_1)\right)\cup \{x_2\}$$
or 
$$S_2'=\left(S_3\setminus  V(G_1)\right)\cup \{y_2\}$$
is a strong dominating set of $G_2$ (or possibly both are strong dominating set of $G_2$). So
$$\gst(G_1) + \gst (G_2) \leq \gst (G_3)+1.$$
\item[(c)]
$y_1\notin S_3$ and $y_2\notin S_3$. Then by considering similar sets as part (a), we have 
$$\gst(G_1) + \gst (G_2) \leq \gst (G_3)+1.$$
\end{itemize}
\end{itemize}
Therefore we have $\gst(G_3)\geq \gst (G_1) + \gst (G_2) -\deg(x_1)-\deg(x_2)+2$, and we are done.
	\end{proof}

\begin{figure}[!h]
	\begin{center}
		\psscalebox{0.6 0.6}
		{
			\begin{pspicture}(0,-12.499306)(13.602778,12.102083)
			\psline[linecolor=black, linewidth=0.08](11.001389,10.700695)(12.201389,11.900695)(12.201389,11.900695)
			\psline[linecolor=black, linewidth=0.08](11.001389,10.700695)(12.201389,11.100695)(12.201389,11.100695)
			\psline[linecolor=black, linewidth=0.08](11.001389,10.700695)(12.201389,10.300694)(12.201389,10.300694)
			\psline[linecolor=black, linewidth=0.08](11.001389,10.700695)(12.201389,9.500694)(12.201389,9.500694)
			\psline[linecolor=black, linewidth=0.08](11.001389,7.5006948)(12.201389,8.700695)(12.201389,8.700695)
			\psline[linecolor=black, linewidth=0.08](11.001389,7.5006948)(12.201389,7.9006944)(12.201389,7.9006944)
			\psline[linecolor=black, linewidth=0.08](11.001389,7.5006948)(12.201389,7.1006947)(12.201389,7.1006947)
			\psline[linecolor=black, linewidth=0.08](11.001389,7.5006948)(12.201389,6.3006945)(12.201389,6.3006945)
			\psline[linecolor=black, linewidth=0.08](11.001389,4.3006945)(12.201389,5.5006948)(12.201389,5.5006948)
			\psline[linecolor=black, linewidth=0.08](11.001389,4.3006945)(12.201389,4.7006946)(12.201389,4.7006946)
			\psline[linecolor=black, linewidth=0.08](11.001389,4.3006945)(12.201389,3.9006946)(12.201389,3.9006946)
			\psline[linecolor=black, linewidth=0.08](11.001389,4.3006945)(12.201389,3.1006947)(12.201389,3.1006947)
			\psline[linecolor=black, linewidth=0.08](11.001389,10.700695)(9.401389,7.5006948)(9.401389,7.5006948)
			\psline[linecolor=black, linewidth=0.08](9.401389,7.5006948)(11.001389,7.5006948)(11.001389,7.5006948)
			\psline[linecolor=black, linewidth=0.08](9.401389,7.5006948)(11.001389,4.3006945)(11.001389,4.3006945)
			\psline[linecolor=black, linewidth=0.08](9.401389,7.5006948)(9.401389,4.3006945)(9.401389,4.3006945)
			\psline[linecolor=black, linewidth=0.08](9.401389,4.3006945)(8.201389,3.5006945)(8.201389,3.5006945)
			\psline[linecolor=black, linewidth=0.08](9.401389,4.3006945)(9.001389,3.5006945)(9.001389,3.5006945)
			\psline[linecolor=black, linewidth=0.08](9.401389,4.3006945)(9.801389,3.5006945)(9.801389,3.5006945)
			\psline[linecolor=black, linewidth=0.08](9.401389,4.3006945)(10.601389,3.5006945)(10.601389,3.5006945)
			\psline[linecolor=black, linewidth=0.08](6.201389,3.5006945)(4.601389,4.3006945)(4.601389,4.3006945)
			\psline[linecolor=black, linewidth=0.08](4.601389,4.3006945)(5.4013886,3.5006945)(5.4013886,3.5006945)
			\psline[linecolor=black, linewidth=0.08](4.601389,4.3006945)(4.601389,3.5006945)(4.601389,3.5006945)
			\psline[linecolor=black, linewidth=0.08](4.601389,4.3006945)(3.8013887,3.5006945)(3.8013887,3.5006945)
			\psline[linecolor=black, linewidth=0.08](4.601389,4.3006945)(3.0013888,3.5006945)(3.0013888,3.5006945)
			\psline[linecolor=black, linewidth=0.08](4.601389,7.5006948)(4.601389,4.3006945)(4.601389,4.3006945)
			\psline[linecolor=black, linewidth=0.08](4.601389,7.5006948)(2.601389,6.3006945)(2.601389,6.3006945)
			\psline[linecolor=black, linewidth=0.08](4.601389,7.5006948)(2.601389,8.700695)(2.601389,8.700695)
			\psline[linecolor=black, linewidth=0.08](2.601389,8.700695)(1.4013889,7.9006944)(1.4013889,7.9006944)
			\psline[linecolor=black, linewidth=0.08](2.601389,8.700695)(1.4013889,8.700695)(1.4013889,8.700695)
			\psline[linecolor=black, linewidth=0.08](2.601389,8.700695)(1.4013889,9.500694)(1.4013889,9.500694)
			\psline[linecolor=black, linewidth=0.08](2.601389,6.3006945)(1.4013889,6.3006945)(1.4013889,6.3006945)
			\psline[linecolor=black, linewidth=0.08](1.4013889,7.1006947)(2.601389,6.3006945)(2.601389,6.3006945)
			\psline[linecolor=black, linewidth=0.08](1.4013889,5.5006948)(2.601389,6.3006945)(2.601389,6.3006945)
			\psline[linecolor=black, linewidth=0.08](13.401389,11.100695)(12.201389,11.100695)(12.201389,11.100695)
			\psline[linecolor=black, linewidth=0.08](13.401389,10.300694)(12.201389,10.300694)(12.201389,10.300694)
			\psline[linecolor=black, linewidth=0.08](13.401389,9.500694)(12.201389,9.500694)(12.201389,9.500694)
			\psline[linecolor=black, linewidth=0.08](12.201389,8.700695)(13.401389,8.700695)(13.401389,8.700695)
			\psline[linecolor=black, linewidth=0.08](12.201389,7.9006944)(13.401389,7.9006944)(13.401389,7.9006944)
			\psline[linecolor=black, linewidth=0.08](12.201389,7.1006947)(13.401389,7.1006947)(13.401389,7.1006947)
			\psline[linecolor=black, linewidth=0.08](12.201389,6.3006945)(13.401389,6.3006945)(13.401389,6.3006945)
			\psline[linecolor=black, linewidth=0.08](12.201389,5.5006948)(13.401389,5.5006948)(13.401389,5.5006948)
			\psline[linecolor=black, linewidth=0.08](12.201389,4.7006946)(13.401389,4.7006946)(13.401389,4.7006946)
			\psline[linecolor=black, linewidth=0.08](12.201389,3.9006946)(13.401389,3.9006946)(13.401389,3.9006946)
			\psline[linecolor=black, linewidth=0.08](12.201389,3.1006947)(13.401389,3.1006947)(13.401389,3.1006947)
			\psline[linecolor=black, linewidth=0.08](1.4013889,9.500694)(0.20138885,9.500694)(0.20138885,9.500694)
			\psline[linecolor=black, linewidth=0.08](0.20138885,8.700695)(1.8013889,8.700695)
			\psline[linecolor=black, linewidth=0.08](0.20138885,7.9006944)(1.4013889,7.9006944)
			\psline[linecolor=black, linewidth=0.08](0.20138885,7.1006947)(1.4013889,7.1006947)
			\psline[linecolor=black, linewidth=0.08](0.20138885,6.3006945)(1.4013889,6.3006945)
			\psline[linecolor=black, linewidth=0.08](0.20138885,5.5006948)(1.4013889,5.5006948)(1.4013889,5.5006948)
			\psline[linecolor=black, linewidth=0.08](3.0013888,3.5006945)(3.0013888,2.3006945)(3.0013888,2.3006945)
			\psline[linecolor=black, linewidth=0.08](3.8013887,3.5006945)(3.8013887,2.3006945)(3.8013887,2.3006945)
			\psline[linecolor=black, linewidth=0.08](4.601389,3.5006945)(4.601389,2.3006945)
			\psline[linecolor=black, linewidth=0.08](5.4013886,3.5006945)(5.4013886,2.3006945)
			\psline[linecolor=black, linewidth=0.08](6.201389,3.5006945)(6.201389,2.3006945)
			\psline[linecolor=black, linewidth=0.08](8.201389,3.5006945)(8.201389,2.3006945)(8.201389,2.3006945)
			\psline[linecolor=black, linewidth=0.08](9.001389,3.5006945)(9.001389,2.3006945)(9.001389,2.3006945)
			\psline[linecolor=black, linewidth=0.08](9.801389,3.5006945)(9.801389,2.3006945)(9.801389,2.3006945)
			\psline[linecolor=black, linewidth=0.08](10.601389,3.5006945)(10.601389,2.3006945)(10.601389,2.3006945)
			\psdots[linecolor=black, dotstyle=o, dotsize=0.4, fillcolor=white](13.401389,11.100695)
			\psdots[linecolor=black, dotstyle=o, dotsize=0.4, fillcolor=white](13.401389,10.300694)
			\psdots[linecolor=black, dotstyle=o, dotsize=0.4, fillcolor=white](13.401389,9.500694)
			\psdots[linecolor=black, dotstyle=o, dotsize=0.4, fillcolor=white](13.401389,8.700695)
			\psdots[linecolor=black, dotstyle=o, dotsize=0.4, fillcolor=white](13.401389,7.9006944)
			\psdots[linecolor=black, dotstyle=o, dotsize=0.4, fillcolor=white](13.401389,7.1006947)
			\psdots[linecolor=black, dotstyle=o, dotsize=0.4, fillcolor=white](13.401389,6.3006945)
			\psdots[linecolor=black, dotstyle=o, dotsize=0.4, fillcolor=white](13.401389,5.5006948)
			\psdots[linecolor=black, dotstyle=o, dotsize=0.4, fillcolor=white](13.401389,4.7006946)
			\psdots[linecolor=black, dotstyle=o, dotsize=0.4, fillcolor=white](13.401389,3.9006946)
			\psdots[linecolor=black, dotstyle=o, dotsize=0.4, fillcolor=white](13.401389,3.1006947)
			\psdots[linecolor=black, dotstyle=o, dotsize=0.4, fillcolor=white](10.601389,2.3006945)
			\psdots[linecolor=black, dotstyle=o, dotsize=0.4, fillcolor=white](9.801389,2.3006945)
			\psdots[linecolor=black, dotstyle=o, dotsize=0.4, fillcolor=white](9.001389,2.3006945)
			\psdots[linecolor=black, dotstyle=o, dotsize=0.4, fillcolor=white](8.201389,2.3006945)
			\psdots[linecolor=black, dotstyle=o, dotsize=0.4, fillcolor=white](6.201389,2.3006945)
			\psdots[linecolor=black, dotstyle=o, dotsize=0.4, fillcolor=white](5.4013886,2.3006945)
			\psdots[linecolor=black, dotstyle=o, dotsize=0.4, fillcolor=white](4.601389,2.3006945)
			\psdots[linecolor=black, dotstyle=o, dotsize=0.4, fillcolor=white](3.8013887,2.3006945)
			\psdots[linecolor=black, dotstyle=o, dotsize=0.4, fillcolor=white](3.0013888,2.3006945)
			\psdots[linecolor=black, dotstyle=o, dotsize=0.4, fillcolor=white](0.20138885,5.5006948)
			\psdots[linecolor=black, dotstyle=o, dotsize=0.4, fillcolor=white](0.20138885,6.3006945)
			\psdots[linecolor=black, dotstyle=o, dotsize=0.4, fillcolor=white](0.20138885,7.1006947)
			\psdots[linecolor=black, dotstyle=o, dotsize=0.4, fillcolor=white](0.20138885,7.9006944)
			\psdots[linecolor=black, dotstyle=o, dotsize=0.4, fillcolor=white](0.20138885,8.700695)
			\psdots[linecolor=black, dotstyle=o, dotsize=0.4, fillcolor=white](0.20138885,9.500694)
			\psdots[linecolor=black, dotsize=0.4](12.201389,11.900695)
			\psdots[linecolor=black, dotsize=0.4](12.201389,11.100695)
			\psdots[linecolor=black, dotsize=0.4](12.201389,10.300694)
			\psdots[linecolor=black, dotsize=0.4](12.201389,9.500694)
			\psdots[linecolor=black, dotsize=0.4](12.201389,8.700695)
			\psdots[linecolor=black, dotsize=0.4](12.201389,7.9006944)
			\psdots[linecolor=black, dotsize=0.4](12.201389,7.1006947)
			\psdots[linecolor=black, dotsize=0.4](12.201389,6.3006945)
			\psdots[linecolor=black, dotsize=0.4](12.201389,5.5006948)
			\psdots[linecolor=black, dotsize=0.4](12.201389,4.7006946)
			\psdots[linecolor=black, dotsize=0.4](12.201389,3.9006946)
			\psdots[linecolor=black, dotsize=0.4](12.201389,3.1006947)
			\psdots[linecolor=black, dotsize=0.4](10.601389,3.5006945)
			\psdots[linecolor=black, dotsize=0.4](9.801389,3.5006945)
			\psdots[linecolor=black, dotsize=0.4](9.001389,3.5006945)
			\psdots[linecolor=black, dotsize=0.4](8.201389,3.5006945)
			\psdots[linecolor=black, dotsize=0.4](6.201389,3.5006945)
			\psdots[linecolor=black, dotsize=0.4](5.4013886,3.5006945)
			\psdots[linecolor=black, dotsize=0.4](4.601389,3.5006945)
			\psdots[linecolor=black, dotsize=0.4](3.8013887,3.5006945)
			\psdots[linecolor=black, dotsize=0.4](3.0013888,3.5006945)
			\psdots[linecolor=black, dotsize=0.4](1.4013889,5.5006948)
			\psdots[linecolor=black, dotsize=0.4](1.4013889,6.3006945)
			\psdots[linecolor=black, dotsize=0.4](1.4013889,7.1006947)
			\psdots[linecolor=black, dotsize=0.4](1.4013889,7.9006944)
			\psdots[linecolor=black, dotsize=0.4](1.4013889,8.700695)
			\psdots[linecolor=black, dotsize=0.4](1.4013889,9.500694)
			\psline[linecolor=black, linewidth=0.08](9.401389,-2.4993055)(10.601389,-1.2993054)(10.601389,-1.2993054)
			\psline[linecolor=black, linewidth=0.08](9.401389,-2.4993055)(10.601389,-2.0993054)(10.601389,-2.0993054)
			\psline[linecolor=black, linewidth=0.08](9.401389,-2.4993055)(10.601389,-2.8993053)(10.601389,-2.8993053)
			\psline[linecolor=black, linewidth=0.08](9.401389,-2.4993055)(10.601389,-3.6993055)(10.601389,-3.6993055)
			\psline[linecolor=black, linewidth=0.08](9.401389,-5.6993055)(10.601389,-4.4993052)(10.601389,-4.4993052)
			\psline[linecolor=black, linewidth=0.08](9.401389,-5.6993055)(10.601389,-5.2993054)(10.601389,-5.2993054)
			\psline[linecolor=black, linewidth=0.08](9.401389,-5.6993055)(10.601389,-6.0993056)(10.601389,-6.0993056)
			\psline[linecolor=black, linewidth=0.08](9.401389,-5.6993055)(10.601389,-6.8993053)(10.601389,-6.8993053)
			\psline[linecolor=black, linewidth=0.08](9.401389,-8.899305)(10.601389,-7.6993055)(10.601389,-7.6993055)
			\psline[linecolor=black, linewidth=0.08](9.401389,-8.899305)(10.601389,-8.499306)(10.601389,-8.499306)
			\psline[linecolor=black, linewidth=0.08](9.401389,-8.899305)(10.601389,-9.299305)(10.601389,-9.299305)
			\psline[linecolor=black, linewidth=0.08](9.401389,-8.899305)(10.601389,-10.099305)(10.601389,-10.099305)
			\psline[linecolor=black, linewidth=0.08](9.401389,-2.4993055)(7.8013887,-5.6993055)(7.8013887,-5.6993055)
			\psline[linecolor=black, linewidth=0.08](7.8013887,-5.6993055)(9.401389,-5.6993055)(9.401389,-5.6993055)
			\psline[linecolor=black, linewidth=0.08](7.8013887,-5.6993055)(9.401389,-8.899305)(9.401389,-8.899305)
			\psline[linecolor=black, linewidth=0.08](7.8013887,-8.899305)(6.601389,-9.699306)(6.601389,-9.699306)
			\psline[linecolor=black, linewidth=0.08](7.8013887,-8.899305)(7.4013886,-9.699306)(7.4013886,-9.699306)
			\psline[linecolor=black, linewidth=0.08](7.8013887,-8.899305)(8.201389,-9.699306)(8.201389,-9.699306)
			\psline[linecolor=black, linewidth=0.08](7.8013887,-8.899305)(9.001389,-9.699306)(9.001389,-9.699306)
			\psline[linecolor=black, linewidth=0.08](5.001389,-9.699306)(3.401389,-8.899305)(3.401389,-8.899305)
			\psline[linecolor=black, linewidth=0.08](3.401389,-8.899305)(4.201389,-9.699306)(4.201389,-9.699306)
			\psline[linecolor=black, linewidth=0.08](3.401389,-8.899305)(3.401389,-9.699306)(3.401389,-9.699306)
			\psline[linecolor=black, linewidth=0.08](3.401389,-8.899305)(2.601389,-9.699306)(2.601389,-9.699306)
			\psline[linecolor=black, linewidth=0.08](3.401389,-8.899305)(1.8013889,-9.699306)(1.8013889,-9.699306)
			\psline[linecolor=black, linewidth=0.08](7.8013887,-5.6993055)(5.8013887,-6.8993053)(5.8013887,-6.8993053)
			\psline[linecolor=black, linewidth=0.08](7.8013887,-5.6993055)(5.8013887,-4.4993052)(5.8013887,-4.4993052)
			\psline[linecolor=black, linewidth=0.08](5.8013887,-4.4993052)(4.601389,-5.2993054)(4.601389,-5.2993054)
			\psline[linecolor=black, linewidth=0.08](5.8013887,-4.4993052)(4.601389,-4.4993052)(4.601389,-4.4993052)
			\psline[linecolor=black, linewidth=0.08](5.8013887,-4.4993052)(4.601389,-3.6993055)(4.601389,-3.6993055)
			\psline[linecolor=black, linewidth=0.08](5.8013887,-6.8993053)(4.601389,-6.8993053)(4.601389,-6.8993053)
			\psline[linecolor=black, linewidth=0.08](4.601389,-6.0993056)(5.8013887,-6.8993053)(5.8013887,-6.8993053)
			\psline[linecolor=black, linewidth=0.08](4.601389,-7.6993055)(5.8013887,-6.8993053)(5.8013887,-6.8993053)
			\psline[linecolor=black, linewidth=0.08](11.801389,-1.2993054)(10.601389,-1.2993054)(10.601389,-1.2993054)
			\psline[linecolor=black, linewidth=0.08](11.801389,-2.0993054)(10.601389,-2.0993054)(10.601389,-2.0993054)
			\psline[linecolor=black, linewidth=0.08](11.801389,-2.8993053)(10.601389,-2.8993053)(10.601389,-2.8993053)
			\psline[linecolor=black, linewidth=0.08](11.801389,-3.6993055)(10.601389,-3.6993055)(10.601389,-3.6993055)
			\psline[linecolor=black, linewidth=0.08](10.601389,-4.4993052)(11.801389,-4.4993052)(11.801389,-4.4993052)
			\psline[linecolor=black, linewidth=0.08](10.601389,-5.2993054)(11.801389,-5.2993054)(11.801389,-5.2993054)
			\psline[linecolor=black, linewidth=0.08](10.601389,-6.0993056)(11.801389,-6.0993056)(11.801389,-6.0993056)
			\psline[linecolor=black, linewidth=0.08](10.601389,-6.8993053)(11.801389,-6.8993053)(11.801389,-6.8993053)
			\psline[linecolor=black, linewidth=0.08](10.601389,-7.6993055)(11.801389,-7.6993055)(11.801389,-7.6993055)
			\psline[linecolor=black, linewidth=0.08](10.601389,-8.499306)(11.801389,-8.499306)(11.801389,-8.499306)
			\psline[linecolor=black, linewidth=0.08](10.601389,-9.299305)(11.801389,-9.299305)(11.801389,-9.299305)
			\psline[linecolor=black, linewidth=0.08](10.601389,-10.099305)(11.801389,-10.099305)(11.801389,-10.099305)
			\psline[linecolor=black, linewidth=0.08](4.601389,-3.6993055)(3.401389,-3.6993055)(3.401389,-3.6993055)
			\psline[linecolor=black, linewidth=0.08](3.401389,-4.4993052)(5.001389,-4.4993052)
			\psline[linecolor=black, linewidth=0.08](3.401389,-5.2993054)(4.601389,-5.2993054)
			\psline[linecolor=black, linewidth=0.08](3.401389,-6.0993056)(4.601389,-6.0993056)
			\psline[linecolor=black, linewidth=0.08](3.401389,-6.8993053)(4.601389,-6.8993053)
			\psline[linecolor=black, linewidth=0.08](3.401389,-7.6993055)(4.601389,-7.6993055)(4.601389,-7.6993055)
			\psline[linecolor=black, linewidth=0.08](1.8013889,-9.699306)(1.8013889,-10.899305)(1.8013889,-10.899305)
			\psline[linecolor=black, linewidth=0.08](2.601389,-9.699306)(2.601389,-10.899305)(2.601389,-10.899305)
			\psline[linecolor=black, linewidth=0.08](3.401389,-9.699306)(3.401389,-10.899305)
			\psline[linecolor=black, linewidth=0.08](4.201389,-9.699306)(4.201389,-10.899305)
			\psline[linecolor=black, linewidth=0.08](5.001389,-9.699306)(5.001389,-10.899305)
			\psline[linecolor=black, linewidth=0.08](6.601389,-9.699306)(6.601389,-10.899305)(6.601389,-10.899305)
			\psline[linecolor=black, linewidth=0.08](7.4013886,-9.699306)(7.4013886,-10.899305)(7.4013886,-10.899305)
			\psline[linecolor=black, linewidth=0.08](8.201389,-9.699306)(8.201389,-10.899305)(8.201389,-10.899305)
			\psline[linecolor=black, linewidth=0.08](9.001389,-9.699306)(9.001389,-10.899305)(9.001389,-10.899305)
			\psdots[linecolor=black, dotstyle=o, dotsize=0.4, fillcolor=white](11.801389,-1.2993054)
			\psdots[linecolor=black, dotstyle=o, dotsize=0.4, fillcolor=white](11.801389,-2.0993054)
			\psdots[linecolor=black, dotstyle=o, dotsize=0.4, fillcolor=white](11.801389,-2.8993053)
			\psdots[linecolor=black, dotstyle=o, dotsize=0.4, fillcolor=white](11.801389,-3.6993055)
			\psdots[linecolor=black, dotstyle=o, dotsize=0.4, fillcolor=white](11.801389,-4.4993052)
			\psdots[linecolor=black, dotstyle=o, dotsize=0.4, fillcolor=white](11.801389,-5.2993054)
			\psdots[linecolor=black, dotstyle=o, dotsize=0.4, fillcolor=white](11.801389,-6.0993056)
			\psdots[linecolor=black, dotstyle=o, dotsize=0.4, fillcolor=white](11.801389,-6.8993053)
			\psdots[linecolor=black, dotstyle=o, dotsize=0.4, fillcolor=white](11.801389,-7.6993055)
			\psdots[linecolor=black, dotstyle=o, dotsize=0.4, fillcolor=white](11.801389,-8.499306)
			\psdots[linecolor=black, dotstyle=o, dotsize=0.4, fillcolor=white](11.801389,-9.299305)
			\psdots[linecolor=black, dotstyle=o, dotsize=0.4, fillcolor=white](11.801389,-10.099305)
			\psdots[linecolor=black, dotstyle=o, dotsize=0.4, fillcolor=white](9.001389,-10.899305)
			\psdots[linecolor=black, dotstyle=o, dotsize=0.4, fillcolor=white](8.201389,-10.899305)
			\psdots[linecolor=black, dotstyle=o, dotsize=0.4, fillcolor=white](7.4013886,-10.899305)
			\psdots[linecolor=black, dotstyle=o, dotsize=0.4, fillcolor=white](6.601389,-10.899305)
			\psdots[linecolor=black, dotstyle=o, dotsize=0.4, fillcolor=white](5.001389,-10.899305)
			\psdots[linecolor=black, dotstyle=o, dotsize=0.4, fillcolor=white](4.201389,-10.899305)
			\psdots[linecolor=black, dotstyle=o, dotsize=0.4, fillcolor=white](3.401389,-10.899305)
			\psdots[linecolor=black, dotstyle=o, dotsize=0.4, fillcolor=white](2.601389,-10.899305)
			\psdots[linecolor=black, dotstyle=o, dotsize=0.4, fillcolor=white](1.8013889,-10.899305)
			\psdots[linecolor=black, dotstyle=o, dotsize=0.4, fillcolor=white](3.401389,-7.6993055)
			\psdots[linecolor=black, dotstyle=o, dotsize=0.4, fillcolor=white](3.401389,-6.8993053)
			\psdots[linecolor=black, dotstyle=o, dotsize=0.4, fillcolor=white](3.401389,-6.0993056)
			\psdots[linecolor=black, dotstyle=o, dotsize=0.4, fillcolor=white](3.401389,-5.2993054)
			\psdots[linecolor=black, dotstyle=o, dotsize=0.4, fillcolor=white](3.401389,-4.4993052)
			\psdots[linecolor=black, dotstyle=o, dotsize=0.4, fillcolor=white](3.401389,-3.6993055)
			\psdots[linecolor=black, dotsize=0.4](10.601389,-1.2993054)
			\psdots[linecolor=black, dotsize=0.4](10.601389,-2.0993054)
			\psdots[linecolor=black, dotsize=0.4](10.601389,-2.8993053)
			\psdots[linecolor=black, dotsize=0.4](10.601389,-3.6993055)
			\psdots[linecolor=black, dotsize=0.4](10.601389,-4.4993052)
			\psdots[linecolor=black, dotsize=0.4](10.601389,-5.2993054)
			\psdots[linecolor=black, dotsize=0.4](10.601389,-6.0993056)
			\psdots[linecolor=black, dotsize=0.4](10.601389,-6.8993053)
			\psdots[linecolor=black, dotsize=0.4](10.601389,-7.6993055)
			\psdots[linecolor=black, dotsize=0.4](10.601389,-8.499306)
			\psdots[linecolor=black, dotsize=0.4](10.601389,-9.299305)
			\psdots[linecolor=black, dotsize=0.4](10.601389,-10.099305)
			\psdots[linecolor=black, dotsize=0.4](9.001389,-9.699306)
			\psdots[linecolor=black, dotsize=0.4](8.201389,-9.699306)
			\psdots[linecolor=black, dotsize=0.4](7.4013886,-9.699306)
			\psdots[linecolor=black, dotsize=0.4](6.601389,-9.699306)
			\psdots[linecolor=black, dotsize=0.4](5.001389,-9.699306)
			\psdots[linecolor=black, dotsize=0.4](4.201389,-9.699306)
			\psdots[linecolor=black, dotsize=0.4](3.401389,-9.699306)
			\psdots[linecolor=black, dotsize=0.4](2.601389,-9.699306)
			\psdots[linecolor=black, dotsize=0.4](1.8013889,-9.699306)
			\psdots[linecolor=black, dotsize=0.4](4.601389,-7.6993055)
			\psdots[linecolor=black, dotsize=0.4](4.601389,-6.8993053)
			\psdots[linecolor=black, dotsize=0.4](4.601389,-6.0993056)
			\psdots[linecolor=black, dotsize=0.4](4.601389,-5.2993054)
			\psdots[linecolor=black, dotsize=0.4](4.601389,-4.4993052)
			\psdots[linecolor=black, dotsize=0.4](4.601389,-3.6993055)
			\psline[linecolor=black, linewidth=0.08](3.401389,-8.899305)(7.8013887,-8.899305)(7.8013887,-8.899305)
			\psdots[linecolor=black, dotsize=0.4](2.601389,8.700695)
			\psdots[linecolor=black, dotsize=0.4](2.601389,6.3006945)
			\psdots[linecolor=black, dotsize=0.4](4.601389,4.3006945)
			\psdots[linecolor=black, dotsize=0.4](9.401389,4.3006945)
			\psdots[linecolor=black, dotsize=0.4](11.001389,4.3006945)
			\psdots[linecolor=black, dotsize=0.4](11.001389,7.5006948)
			\psdots[linecolor=black, dotsize=0.4](11.001389,10.700695)
			\psdots[linecolor=black, dotsize=0.4](7.8013887,-5.6993055)
			\psdots[linecolor=black, dotsize=0.4](3.401389,-8.899305)
			\psdots[linecolor=black, dotstyle=o, dotsize=0.4, fillcolor=white](9.401389,-2.4993055)
			\psdots[linecolor=black, dotstyle=o, dotsize=0.4, fillcolor=white](9.401389,-5.6993055)
			\psdots[linecolor=black, dotstyle=o, dotsize=0.4, fillcolor=white](9.401389,-8.899305)
			\psdots[linecolor=black, dotstyle=o, dotsize=0.4, fillcolor=white](5.8013887,-4.4993052)
			\psdots[linecolor=black, dotstyle=o, dotsize=0.4, fillcolor=white](5.8013887,-6.8993053)
			\psdots[linecolor=black, dotstyle=o, dotsize=0.4, fillcolor=white](7.8013887,-8.899305)
			\psdots[linecolor=black, dotstyle=o, dotsize=0.4, fillcolor=white](9.401389,7.5006948)
			\psdots[linecolor=black, dotstyle=o, dotsize=0.4, fillcolor=white](4.601389,7.5006948)
			\rput[bl](4.201389,-12.499306){\LARGE{$G_1(x_1y_1)+_H G_2(x_2y_2)$}}
			\rput[bl](10.601389,0.70069456){\LARGE{$G_2$}}
			\rput[bl](2.601389,0.70069456){\LARGE{$G_1$}}
			\psrotate(5.981389, -5.8793054){0.25591654}{\rput[bl](5.981389,-5.8793054){$v_H(x_1x_2)$}}
			\rput[bl](4.9213886,7.3606944){$x_1$}
			\rput[bl](8.601389,7.3806944){$x_2$}
			\rput[bl](4.8813887,4.460695){$y_1$}
			\rput[bl](2.881389,-8.619306){$y_1$}
			\rput[bl](8.721389,4.4006944){$y_2$}
			\rput[bl](7.641389,-8.539306){$y_2$}
			\psline[linecolor=black, linewidth=0.08](12.201389,11.900695)(13.401389,11.900695)(13.401389,11.900695)
			\psdots[linecolor=black, dotstyle=o, dotsize=0.4, fillcolor=white](13.401389,11.900695)
			\end{pspicture}
		}
	\end{center}
	\caption{Haj\'{o}s construction of $G_1$ and $G_2$.} \label{fig:hajos-low}
\end{figure}

\begin{figure}[!h]
	\begin{center}
		\psscalebox{0.6 0.6}
		{
			\begin{pspicture}(0,-7.7393055)(18.402779,-2.2579165)
			\rput[bl](12.101389,-7.7393055){\LARGE{$G_1(x_1y_1)+_H G_2(x_2y_2)$}}
			\rput[bl](5.8013887,-7.6593056){\LARGE{$G_2$}}
			\rput[bl](1.0013889,-7.6593056){\LARGE{$G_1$}}
			\psrotate(13.901389, -3.5993054){0.25591654}{\rput[bl](13.901389,-3.5993054){$v_H(x_1x_2)$}}
			\rput[bl](3.061389,-3.7993054){$x_1$}
			\rput[bl](4.161389,-3.8193054){$x_2$}
			\rput[bl](2.881389,-6.9393053){$y_1$}
			\rput[bl](4.2813888,-6.9393053){$y_2$}
			\rput[bl](13.321389,-6.7193055){$y_1$}
			\rput[bl](15.5413885,-6.7193055){$y_2$}
			\psline[linecolor=black, linewidth=0.08](4.601389,-4.059305)(4.601389,-6.4593053)(4.601389,-6.4593053)
			\psline[linecolor=black, linewidth=0.08](4.601389,-4.059305)(6.201389,-4.059305)(6.201389,-4.059305)
			\psline[linecolor=black, linewidth=0.08](4.601389,-4.059305)(6.201389,-2.4593055)(6.201389,-2.4593055)
			\psline[linecolor=black, linewidth=0.08](4.601389,-4.059305)(6.201389,-5.6593056)(6.201389,-5.6593056)
			\psline[linecolor=black, linewidth=0.08](4.601389,-6.4593053)(6.201389,-4.059305)(6.201389,-4.059305)
			\psline[linecolor=black, linewidth=0.08](6.201389,-2.4593055)(6.201389,-4.059305)(6.201389,-4.059305)
			\psline[linecolor=black, linewidth=0.08](6.201389,-4.059305)(6.201389,-5.6593056)(6.201389,-5.6593056)
			\psline[linecolor=black, linewidth=0.08](6.201389,-5.6593056)(7.4013886,-5.6593056)(7.4013886,-5.6593056)
			\psline[linecolor=black, linewidth=0.08](6.201389,-4.059305)(7.4013886,-4.059305)(7.4013886,-4.059305)
			\psline[linecolor=black, linewidth=0.08](6.201389,-2.4593055)(7.4013886,-2.4593055)(7.4013886,-2.4593055)
			\psdots[linecolor=black, dotstyle=o, dotsize=0.4, fillcolor=white](7.4013886,-2.4593055)
			\psdots[linecolor=black, dotstyle=o, dotsize=0.4, fillcolor=white](7.4013886,-4.059305)
			\psdots[linecolor=black, dotstyle=o, dotsize=0.4, fillcolor=white](7.4013886,-5.6593056)
			\psline[linecolor=black, linewidth=0.08](3.0013888,-4.059305)(3.0013888,-6.4593053)(3.0013888,-6.4593053)
			\psline[linecolor=black, linewidth=0.08](3.0013888,-4.059305)(1.4013889,-4.059305)(1.4013889,-4.059305)
			\psline[linecolor=black, linewidth=0.08](1.4013889,-4.059305)(3.0013888,-6.4593053)(3.0013888,-6.4593053)
			\psline[linecolor=black, linewidth=0.08](3.0013888,-4.059305)(1.4013889,-2.4593055)(1.4013889,-2.4593055)
			\psline[linecolor=black, linewidth=0.08](1.4013889,-2.4593055)(1.4013889,-4.059305)(1.4013889,-4.059305)
			\psline[linecolor=black, linewidth=0.08](1.4013889,-2.4593055)(0.20138885,-2.4593055)(0.20138885,-2.4593055)
			\psline[linecolor=black, linewidth=0.08](0.20138885,-4.059305)(1.4013889,-4.059305)(1.4013889,-4.059305)
			\psline[linecolor=black, linewidth=0.08](1.4013889,-4.059305)(1.4013889,-5.6593056)(1.4013889,-5.6593056)
			\psline[linecolor=black, linewidth=0.08](1.4013889,-5.6593056)(0.20138885,-5.6593056)(0.20138885,-5.6593056)
			\psline[linecolor=black, linewidth=0.08](3.0013888,-4.059305)(1.4013889,-5.6593056)(1.4013889,-5.6593056)
			\psdots[linecolor=black, dotstyle=o, dotsize=0.4, fillcolor=white](0.20138885,-2.4593055)
			\psdots[linecolor=black, dotstyle=o, dotsize=0.4, fillcolor=white](0.20138885,-4.059305)
			\psdots[linecolor=black, dotstyle=o, dotsize=0.4, fillcolor=white](0.20138885,-5.6593056)
			\psline[linecolor=black, linewidth=0.08](15.401389,-6.4593053)(17.001389,-4.059305)(17.001389,-4.059305)
			\psline[linecolor=black, linewidth=0.08](17.001389,-2.4593055)(17.001389,-4.059305)(17.001389,-4.059305)
			\psline[linecolor=black, linewidth=0.08](17.001389,-4.059305)(17.001389,-5.6593056)(17.001389,-5.6593056)
			\psline[linecolor=black, linewidth=0.08](17.001389,-5.6593056)(18.20139,-5.6593056)(18.20139,-5.6593056)
			\psline[linecolor=black, linewidth=0.08](17.001389,-4.059305)(18.20139,-4.059305)(18.20139,-4.059305)
			\psline[linecolor=black, linewidth=0.08](17.001389,-2.4593055)(18.20139,-2.4593055)(18.20139,-2.4593055)
			\psdots[linecolor=black, dotstyle=o, dotsize=0.4, fillcolor=white](18.20139,-2.4593055)
			\psdots[linecolor=black, dotstyle=o, dotsize=0.4, fillcolor=white](18.20139,-4.059305)
			\psdots[linecolor=black, dotstyle=o, dotsize=0.4, fillcolor=white](18.20139,-5.6593056)
			\psline[linecolor=black, linewidth=0.08](12.201389,-4.059305)(13.801389,-6.4593053)(13.801389,-6.4593053)
			\psline[linecolor=black, linewidth=0.08](12.201389,-2.4593055)(12.201389,-4.059305)(12.201389,-4.059305)
			\psline[linecolor=black, linewidth=0.08](12.201389,-2.4593055)(11.001389,-2.4593055)(11.001389,-2.4593055)
			\psline[linecolor=black, linewidth=0.08](11.001389,-4.059305)(12.201389,-4.059305)(12.201389,-4.059305)
			\psline[linecolor=black, linewidth=0.08](12.201389,-4.059305)(12.201389,-5.6593056)(12.201389,-5.6593056)
			\psline[linecolor=black, linewidth=0.08](12.201389,-5.6593056)(11.001389,-5.6593056)(11.001389,-5.6593056)
			\psdots[linecolor=black, dotstyle=o, dotsize=0.4, fillcolor=white](11.001389,-2.4593055)
			\psdots[linecolor=black, dotstyle=o, dotsize=0.4, fillcolor=white](11.001389,-4.059305)
			\psdots[linecolor=black, dotstyle=o, dotsize=0.4, fillcolor=white](11.001389,-5.6593056)
			\psdots[linecolor=black, dotsize=0.4](17.001389,-2.4593055)
			\psdots[linecolor=black, dotsize=0.4](17.001389,-4.059305)
			\psdots[linecolor=black, dotsize=0.4](17.001389,-5.6593056)
			\psdots[linecolor=black, dotsize=0.4](12.201389,-2.4593055)
			\psdots[linecolor=black, dotsize=0.4](12.201389,-4.059305)
			\psdots[linecolor=black, dotsize=0.4](12.201389,-5.6593056)
			\psdots[linecolor=black, dotsize=0.4](1.4013889,-2.4593055)
			\psdots[linecolor=black, dotsize=0.4](1.4013889,-4.059305)
			\psdots[linecolor=black, dotsize=0.4](1.4013889,-5.6593056)
			\psdots[linecolor=black, dotsize=0.4](6.201389,-2.4593055)
			\psdots[linecolor=black, dotsize=0.4](6.201389,-4.059305)
			\psdots[linecolor=black, dotsize=0.4](6.201389,-5.6593056)
			\psdots[linecolor=black, dotsize=0.4](14.601389,-4.059305)
			\psdots[linecolor=black, dotstyle=o, dotsize=0.4, fillcolor=white](3.0013888,-4.059305)
			\psdots[linecolor=black, dotstyle=o, dotsize=0.4, fillcolor=white](4.601389,-4.059305)
			\psdots[linecolor=black, dotstyle=o, dotsize=0.4, fillcolor=white](3.0013888,-6.4593053)
			\psdots[linecolor=black, dotstyle=o, dotsize=0.4, fillcolor=white](4.601389,-6.4593053)
			\psline[linecolor=black, linewidth=0.08](13.801389,-6.4593053)(15.401389,-6.4593053)(15.401389,-6.4593053)
			\psdots[linecolor=black, dotstyle=o, dotsize=0.4, fillcolor=white](13.801389,-6.4593053)
			\psdots[linecolor=black, dotstyle=o, dotsize=0.4, fillcolor=white](15.401389,-6.4593053)
			\psline[linecolor=black, linewidth=0.08](12.201389,-2.4593055)(14.601389,-4.059305)(14.601389,-4.059305)
			\psline[linecolor=black, linewidth=0.08](12.201389,-4.059305)(14.601389,-4.059305)(14.601389,-4.059305)
			\psline[linecolor=black, linewidth=0.08](12.201389,-5.6593056)(14.601389,-4.059305)(14.601389,-4.059305)
			\psline[linecolor=black, linewidth=0.08](17.001389,-2.4593055)(14.601389,-4.059305)(14.601389,-4.059305)
			\psline[linecolor=black, linewidth=0.08](14.601389,-4.059305)(17.001389,-4.059305)(17.001389,-4.059305)
			\psline[linecolor=black, linewidth=0.08](14.601389,-4.059305)(17.001389,-5.6593056)(17.001389,-5.6593056)
			\end{pspicture}
		}
	\end{center}
	\caption{Haj\'{o}s construction of $G_1$ and $G_2$.} \label{fig:hajos-up}
\end{figure}

\begin{remark}
	The lower bounds in Theorem \ref{thm:Hajos} is tight. 	Consider Figure \ref{fig:hajos-low}. One can easily check that the set of black vertices in each graph is a strong dominating set of that and the equality holds. This idea can be generalized and therefore there is an infinite family of graphs such that the equality of the lower bound holds. Also, the upper bounds in Theorem \ref{thm:Hajos} is tight. 	Consider Figure \ref{fig:hajos-up}. By an easy argument, the set of black vertices in each graph is a strong dominating set of that and the equality of the upper bound holds. Since this idea can be generalized, then there is an infinite family of graphs such that the equality of the upper bound holds.
\end{remark}

\section{Vertex-Sum}
In this section, we focus on the strong domination number  of vertex-sum graphs. Given  disjoint graphs $G_1,\ldots,G_k$ with $u_i\in V(G_i)$, $i=1,\ldots,k$, the vertex-sum of $G_1,\ldots,G_k$,  at the vertices $u_1,\ldots,u_k$, is the graph $G_1\stackplus{u} G_2 \stackplus{u} \cdots \stackplus{u}  G_k$ obtained from $G_1,\ldots,G_k$ by identifying the vertices $u_i$, $i=1,\ldots,k$, as the same vertex $u$. This definition is from~\cite{Barioli2004} by Barioli, Fallat and Hogben. We call $u$ the \textit{central vertex} of the vertex-sum.	The vertex-sum of $t$ copies of a graph $G$ at a vertex $u$ is denoted by $G_u^t$, $t \geq 2$.	For the sake of simplicity,	we may assume that the vertex $u$ belongs to all the  $G_i$. Recently the distinguishing number and
the distinguishing threshold of some vertex-sum graphs studied in \cite{IJST}. The following theorem gives the lower bound and the upper bound for the strong domination number of vertex-sum of two graphs.  

	\begin{theorem}\label{thm:v-sum}
For the vertex-sum of disjoint graphs $G_1,G_2,\ldots,G_k$ with $u_i\in V(G_i)$, $i=1,2,\ldots,k$, we have
$$\left( \sum_{i=1}^{k}\gst(G_i)-\deg(u_i)\right) +1 \leq  \gst(G_1\underset{u}{+} G_2 \underset{u}{+} \ldots \underset{u}{+} G_k) 
\leq   \left( \sum_{i=1}^{k}\gst(G_i)\right)+1.  $$	
	\end{theorem}

	\begin{proof}
First we find the upper bound. Suppose that $D_i$ is a $\gst$-set of $G_i$, for $i=1,2,\ldots,k$. Then clearly 
$$D=\bigcup\limits_{i=1}^{k} D_i \cup \{u\},$$
is a strong dominating set of $G_1\underset{u}{+} G_2 \underset{u}{+} \ldots \underset{u}{+} G_k$, and we are done. 
Now,  we consider the lower bound and prove it. Suppose that $S$ is a $\gst$-set of $G_1\underset{u}{+} G_2 \underset{u}{+} \ldots \underset{u}{+} G_k$. We find strong dominating sets of $G_i$, for $i=1,2,\ldots,k$, based on $S$. We have two cases:
\begin{itemize}
\item[(i)]
$u\notin S$. Then there exists $u'\in S$ and is strong dominating $u$. Without loss of generality, suppose that $u'\in V(G_1)$. Then one can easily check that 
$$S_1=S\setminus \left(\bigcup\limits_{i=2}^{k} V(G_i)\right)$$ 
is a strong dominating set of $G_1$, and for  $i=2,3,\ldots,k$, 
$$S_i=\Big(S\cup\{u_i\}\Big)\setminus \left(\bigcup\limits_{\underset{j\neq i}{j=1}}^{k} V(G_j)\right) $$
is a strong dominating set of $G_i$. So we have 
$$ \sum_{i=1}^{k}\gst(G_i) \leq  \gst(G_1\underset{u}{+} G_2 \underset{u}{+} \ldots \underset{u}{+} G_k)+k-1,$$
which is not in contradiction of the lower bound.  
\item[(ii)]
$u\in S$. If after forming each $G_i$, for all $i=1,2,\ldots,k$, $\deg(u_i)\geq \max \{\deg(v)~|~ v\in N(u_i)\}$, then 
$$S_i=\Big(S\cup\{u_i\}\Big)\setminus \left(\bigcup\limits_{\underset{j\neq i}{j=1}}^{k} V(G_j)\cup\{u\}\right) $$
is a strong dominating set of $G_i$, for $i=1,2,\ldots,k$. So we have 
$$ \sum_{i=1}^{k}\gst(G_i) \leq  \gst(G_1\underset{u}{+} G_2 \underset{u}{+} \ldots \underset{u}{+} G_k)+k-1,$$
which is not in contradiction of the lower bound. The worst case happens when after forming each $G_i$, for all $i=1,2,\ldots,k$, $\deg(u_i) < \max \{\deg(v)~|~ v\in N(u_i)\}$.  Then by considering
$$S_i=\Big(S\cup N(u_i)\Big)\setminus \left(\bigcup\limits_{\underset{j\neq i}{j=1}}^{k} V(G_j)\cup\{u\}\right), $$
one can easily check that $S_i$ is a a strong dominating set of $G_i$, for $i=1,2,\ldots,k$. So we have
$$ \sum_{i=1}^{k}\gst(G_i) \leq  \gst(G_1\underset{u}{+} G_2 \underset{u}{+} \ldots \underset{u}{+} G_k)+\left(\sum_{i=1}^{k}\deg(u_i)\right) -1.$$
\end{itemize}
Therefore we have the result.
	\end{proof}

As an immediate result of Theorem \ref{thm:v-sum}, we have:

\begin{corollary}
For the The vertex-sum of $t$ copies of a graph $G$ at a vertex $u$, we have 
$$t\big(\gst(G)-\deg(u)\big) +1 \leq  \gst(G_u^t) \leq   t\gst(G)+1.  $$	
\end{corollary}

\begin{remark}
Bounds in Theorem \ref{thm:v-sum} are tight. For the upper bound, consider $G_i$ as shown in Figure \ref{fig:v-sum-upper}. The set of black vertices is a $\gst$-set of $G_i$. Now, if we consider $G_1\underset{u}{+} G_2 \underset{u}{+} \ldots \underset{u}{+} G_k$, then we need all black vertices and $u$ in our strong dominating set. Therefore the equality holds. By generalizing this idea, we have an infinite family of graphs such that the equality of the upper bound holds. For the lower bound, consider $G_i$ as shown in Figure \ref{fig:v-sum-lower}. The set of black vertices, say $S_i$, is a $\gst$-set of $G_i$. Now, if we consider $G_1\underset{u}{+} G_2 \underset{u}{+} \ldots \underset{u}{+} G_k$, then clearly $ \left(\bigcup\limits_{i=1}^{k} S_i\cup\{u\}\right)\setminus \left(\bigcup\limits_{i=1}^{k} N(u_i)\right)$ is a $\gst$-set, and we are done. By generalizing this idea, we have an infinite family of graphs such that the equality of the lower bound holds.
\end{remark}

\begin{figure}
\begin{center}
\psscalebox{0.6 0.6}
{
\begin{pspicture}(0,-5.025)(6.8027782,-0.155)
\psline[linecolor=black, linewidth=0.08](0.20138885,-3.965)(1.4013889,-2.765)(2.601389,-3.965)(2.601389,-3.965)
\psline[linecolor=black, linewidth=0.08](1.4013889,-2.765)(3.401389,-0.765)(5.4013886,-2.765)(5.4013886,-2.765)
\psline[linecolor=black, linewidth=0.08](5.4013886,-2.765)(4.201389,-3.965)(4.201389,-3.965)
\psline[linecolor=black, linewidth=0.08](5.4013886,-2.765)(6.601389,-3.965)(6.601389,-3.965)
\psdots[linecolor=black, dotsize=0.4](1.4013889,-2.765)
\psdots[linecolor=black, dotsize=0.4](5.4013886,-2.765)
\psdots[linecolor=black, dotstyle=o, dotsize=0.4, fillcolor=white](0.20138885,-3.965)
\psdots[linecolor=black, dotstyle=o, dotsize=0.4, fillcolor=white](2.601389,-3.965)
\psdots[linecolor=black, dotstyle=o, dotsize=0.4, fillcolor=white](4.201389,-3.965)
\psdots[linecolor=black, dotstyle=o, dotsize=0.4, fillcolor=white](6.601389,-3.965)
\psdots[linecolor=black, dotstyle=o, dotsize=0.4, fillcolor=white](3.401389,-0.765)
\rput[bl](3.2013888,-0.405){$u_i$}
\rput[bl](3.2013888,-5.025){\Large{$G_i$}}
\end{pspicture}
}
\end{center}
\caption{Graph $G_i$, for $i=1,2,\ldots,k$.} \label{fig:v-sum-upper}
\end{figure}

\begin{figure}
\begin{center}
\psscalebox{0.6 0.6}
{
\begin{pspicture}(0,-8.155)(9.202778,-2.665)
\rput[bl](4.441389,-2.915){$u_i$}
\rput[bl](4.461389,-8.155){\Large{$G_i$}}
\psline[linecolor=black, linewidth=0.08](1.4013889,-4.855)(0.20138885,-6.055)(0.20138885,-6.055)
\psline[linecolor=black, linewidth=0.08](1.4013889,-4.855)(1.0013889,-6.055)(1.0013889,-6.055)
\psline[linecolor=black, linewidth=0.08](1.4013889,-4.855)(1.8013889,-6.055)(1.8013889,-6.055)
\psline[linecolor=black, linewidth=0.08](1.4013889,-4.855)(2.601389,-6.055)(2.601389,-6.055)
\psline[linecolor=black, linewidth=0.08](4.601389,-4.855)(3.401389,-6.055)(3.401389,-6.055)
\psline[linecolor=black, linewidth=0.08](4.601389,-4.855)(4.201389,-6.055)(4.201389,-6.055)(4.201389,-6.055)
\psline[linecolor=black, linewidth=0.08](4.601389,-4.855)(5.001389,-6.055)(5.001389,-6.055)
\psline[linecolor=black, linewidth=0.08](4.601389,-4.855)(5.8013887,-6.055)(5.8013887,-6.055)
\psline[linecolor=black, linewidth=0.08](7.8013887,-4.855)(6.601389,-6.055)(6.601389,-6.055)
\psline[linecolor=black, linewidth=0.08](7.8013887,-4.855)(7.4013886,-6.055)(7.4013886,-6.055)
\psline[linecolor=black, linewidth=0.08](7.8013887,-4.855)(8.201389,-6.055)(8.201389,-6.055)
\psline[linecolor=black, linewidth=0.08](7.8013887,-4.855)(9.001389,-6.055)(9.001389,-6.055)
\psline[linecolor=black, linewidth=0.08](0.20138885,-6.055)(0.20138885,-7.255)(0.20138885,-7.255)
\psline[linecolor=black, linewidth=0.08](1.0013889,-6.055)(1.0013889,-7.255)(1.0013889,-7.255)
\psline[linecolor=black, linewidth=0.08](1.8013889,-6.055)(1.8013889,-7.255)(1.8013889,-7.255)
\psline[linecolor=black, linewidth=0.08](2.601389,-6.055)(2.601389,-7.255)(2.601389,-7.255)
\psline[linecolor=black, linewidth=0.08](3.401389,-6.055)(3.401389,-7.255)(3.401389,-7.255)
\psline[linecolor=black, linewidth=0.08](4.201389,-6.055)(4.201389,-7.255)(4.201389,-7.255)
\psline[linecolor=black, linewidth=0.08](5.001389,-6.055)(5.001389,-7.255)(5.001389,-7.255)
\psline[linecolor=black, linewidth=0.08](5.8013887,-6.055)(5.8013887,-7.255)(5.8013887,-7.255)
\psline[linecolor=black, linewidth=0.08](4.601389,-4.855)(4.601389,-3.255)(4.601389,-3.255)
\psline[linecolor=black, linewidth=0.08](1.4013889,-4.855)(4.601389,-3.255)(4.601389,-3.255)
\psline[linecolor=black, linewidth=0.08](4.601389,-3.255)(7.8013887,-4.855)(7.8013887,-4.855)
\psline[linecolor=black, linewidth=0.08](6.601389,-7.255)(6.601389,-6.055)(6.601389,-6.055)
\psline[linecolor=black, linewidth=0.08](7.4013886,-6.055)(7.4013886,-7.255)
\psline[linecolor=black, linewidth=0.08](8.201389,-6.055)(8.201389,-7.255)
\psline[linecolor=black, linewidth=0.08](9.001389,-6.055)(9.001389,-7.255)(9.001389,-7.255)
\psdots[linecolor=black, dotstyle=o, dotsize=0.4, fillcolor=white](4.601389,-3.255)
\psdots[linecolor=black, dotstyle=o, dotsize=0.4, fillcolor=white](0.20138885,-7.255)
\psdots[linecolor=black, dotstyle=o, dotsize=0.4, fillcolor=white](1.0013889,-7.255)
\psdots[linecolor=black, dotstyle=o, dotsize=0.4, fillcolor=white](1.8013889,-7.255)
\psdots[linecolor=black, dotstyle=o, dotsize=0.4, fillcolor=white](2.601389,-7.255)
\psdots[linecolor=black, dotstyle=o, dotsize=0.4, fillcolor=white](3.401389,-7.255)
\psdots[linecolor=black, dotstyle=o, dotsize=0.4, fillcolor=white](4.201389,-7.255)
\psdots[linecolor=black, dotstyle=o, dotsize=0.4, fillcolor=white](5.001389,-7.255)
\psdots[linecolor=black, dotstyle=o, dotsize=0.4, fillcolor=white](5.8013887,-7.255)
\psdots[linecolor=black, dotstyle=o, dotsize=0.4, fillcolor=white](6.601389,-7.255)
\psdots[linecolor=black, dotstyle=o, dotsize=0.4, fillcolor=white](7.4013886,-7.255)
\psdots[linecolor=black, dotstyle=o, dotsize=0.4, fillcolor=white](8.201389,-7.255)
\psdots[linecolor=black, dotstyle=o, dotsize=0.4, fillcolor=white](9.001389,-7.255)
\psdots[linecolor=black, dotsize=0.4](1.4013889,-4.855)
\psdots[linecolor=black, dotsize=0.4](4.601389,-4.855)
\psdots[linecolor=black, dotsize=0.4](7.8013887,-4.855)
\psdots[linecolor=black, dotsize=0.4](0.20138885,-6.055)
\psdots[linecolor=black, dotsize=0.4](1.0013889,-6.055)
\psdots[linecolor=black, dotsize=0.4](1.8013889,-6.055)
\psdots[linecolor=black, dotsize=0.4](2.601389,-6.055)
\psdots[linecolor=black, dotsize=0.4](3.401389,-6.055)
\psdots[linecolor=black, dotsize=0.4](4.201389,-6.055)
\psdots[linecolor=black, dotsize=0.4](5.001389,-6.055)
\psdots[linecolor=black, dotsize=0.4](5.8013887,-6.055)
\psdots[linecolor=black, dotsize=0.4](6.601389,-6.055)
\psdots[linecolor=black, dotsize=0.4](7.4013886,-6.055)
\psdots[linecolor=black, dotsize=0.4](8.201389,-6.055)
\psdots[linecolor=black, dotsize=0.4](9.001389,-6.055)
\end{pspicture}
}
\end{center}
\caption{Graph $G_i$, for $i=1,2,\ldots,k$.} \label{fig:v-sum-lower}
\end{figure}



\begin{thebibliography}{99}


	\bibitem{euro}  S. Akbari,  S. Alikhani, Y.H. Peng, {   Characterization of
		graphs using domination polynomial},  {\it Europ. J. Combin.},  {\bf 31} (2010)  1714-1724.
	
	
	\bibitem{Emeric} S. Alikhani, E. Deutsch, More on domination polynomial and domination root, {\it Ars Combin.}  {\bf 134} (2017) 215--232. 
	
	\bibitem{sub} S. Alikhani, N. Ghanbari, H. Zaherifar, Strong domination number of some operations on a graph, submitted. Available at
	\texttt{https://arxiv.org/abs/2210.11120.}
	
	
	
	

	
	
	\bibitem{saeid1} S. Alikhani, Y.H. Peng, {  Introduction to domination polynomial of a graph}, {\it Ars Combin}. {\bf 114} (2014)  257-266.
	
\bibitem{Barioli2004}	F. Barioli, S. Fallat, and L. Hogben. Computation of minimal rank and path cover number
	for certain graphs, {\it Linear Algebra Appl}. {\bf 392} (2004) 289-–303. 
	
	
	\bibitem{Boutrig} R. Boutrig, M. Chellali, { A note on a relation between the weak and strong
		domination numbers of a graph}, {\it Opuscula Math}. {\bf 32} (2012) 235-238.
	
	
	
	\bibitem{HAJOSSUM} G. Haj\'{o}s, {\"U}ber eine Konstruktion nicht $n$-f{\"a}rbbarer Graphen, {\it Wiss. Z. Martin-Luther-Univ. Halle-Wittenberg Math.-Natur. Reihe\/}, 10 (1961) 116--117.
	
	
	
	\bibitem{domination}  T.W. Haynes, S.T. Hedetniemi, P.J. Slater,   {\it Fundamentals of domination in graphs}, Marcel Dekker, NewYork  (1998).
	
	
	\bibitem{DM2}   D. Rautenbach,   {Bounds on the strong domination number graphs}, {\it Discrete Math.}, {\bf 215} (2000) 201-212.
	
	
	
	\bibitem{DM} E. Sampathkumar, L.Pushpa Latha, {Strong weak domination and domination balance in a graph}, {\it Discrete Math}. {\bf 161} (1) (1996) 235-242.
	
	\bibitem{IJST} M.H. Shekarriz, S.A. Talebpour, B. Ahmadi, M.H. Shirdareh Haghighi, S. Alikhani, Distinguishing threshold for some graph operations, {\it Iran J. Sci. Technol. Trans. Sci.}, \texttt{https://doi.org/10.1007/s40995-022-01379-2}. 
	
	\bibitem{utilitas} S.K. Vaidya, R.N. Mehta,  { Strong domination number of some cycle related graphs,
	}		{\it Int. J. Math.} {\bf 3} (2017)  72-80.
	
	\bibitem{weakly} S.K. Vaidya, S.H. Karkar, {On Strong domination number of graphs},  {\it Saurashtra University, India} {\bf 12} (2017) 604-612.
	
	\bibitem{JAS}  H. Zaherifar, S. Alikhani, N. Ghanbari, { On the strong dominating sets of graphs}, {\it J. Alg. Sys.}, {\bf 11} (1) (2023) 65-76. 


\end{thebibliography}
\end{document}